\documentclass{amsart}
\usepackage{amsmath}
  \usepackage{paralist}
  \usepackage{graphics} 
  \usepackage{epsfig} 
\usepackage{graphicx}  \usepackage{epstopdf}
 \usepackage[colorlinks=true]{hyperref}
\hypersetup{urlcolor=blue, citecolor=red}

\newtheorem{theorem}{Theorem}[section]
\newtheorem{lemma}[theorem]{Lemma}
\newtheorem{example}[theorem]{Example}

\newtheorem{remark}[theorem]{Remark}

\newcommand{\cL}{\mathcal{L}}
\newcommand{\cM}{\mathcal{M}}

\newcommand{\cR}{\mathcal{R}}
\newcommand{\sech}{\mbox{sech}}

\title{Stability of the solitary wave solutions to a coupled BBM system}

\author{Hongqiu Chen and Xiao-Jun Wang*}

 \email{hchen1@memphis.edu}
 \email{xwang13@memphis.edu}

\begin{document}
\maketitle
\centerline{\scshape Hongqiu Chen}
\medskip
{\footnotesize
 \centerline{Department of mathematical sciences, University of Memphis, TN, 38152}
} 

\medskip

\centerline{\scshape Xiao-Jun Wang}
\medskip
{\footnotesize
 \centerline{Department of mathematical sciences, University of Memphis, TN, 38152}
}

\bigskip


\begin{abstract}
In this work, we present a stability criteria for the solitary wave solutions to a BBM system that contains coupled nonlinear terms. Using the idea by Bona, Chen and Karakashian \cite{BCK15} and exploiting the accurate point spectrum information of the associated Schr$\ddot{o}$dinger operator, we improve the stability results previously got by Pereira \cite{Per05}.
\end{abstract}

\section{Introduction}

This work is essentially motivated by two papers, \cite{BCK15} and \cite{Per05}. In \cite{BCK15}, Bona, Chen and Karakashian consider a class of BBM-type systems in form of
\begin{eqnarray}\label{main00}
&&u_t+u_x-u_{xxt}+P(u,v)_x=0, \nonumber\\
&&v_t+v_x-v_{xxt}+Q(u,v)_x =0,
\end{eqnarray}
where 
$$P(u, v)=Au^2+Buv+Cv^2, \quad  Q(u, v)=Du^2+Euv+Fv^2$$
with constants $A, B, \cdots, F\in \mathbb R$. Namely, $P$ and $Q$ are homogeneous, quadratic polynomials. 
The other paper \cite{Per05} by Pereira considers BBM system
\begin{eqnarray}\label{main0}
U_t+c_0U_x-AU_{xxt}+(\nabla H(u))_x=0,
\end{eqnarray} where $U=U(x,t)$ is an $\cR^2$-valued function, $c_0$ is a non-negative parameter, $A$ is a $2\times 2$ real positive definite matrix, and $\nabla H$ is the gradient of a homogeneous function $H:\cR^2\rightarrow \cR$ with proper regularity. The nonlinear terms in (\ref{main0}) can be seen as generalizations of that in (\ref{main00}), since the latter were essentially forced into a structure of the same form as the former, due to certain positiveness requirement. However, the result in \cite{BCK15} turns out to be stronger than that in \cite{Per05}. The reason is that, in a proper space, \cite{BCK15} takes into account the accurate point spectrum information of the Schr$\ddot{o}$dinger operator $\cL_0=-{d^2\over dx^2}-\phi(x)$, where $\phi$ is a typical potential function. Here we present a stability result which generalizes \cite{BCK15} and improves \cite{Per05}.
\subsection{Main result}
We consider a BBM system in form of
\begin{eqnarray}\label{main1}
U_t+U_x-U_{xxt}+(\nabla H(U))_x=0,
\end{eqnarray} where $U=U(x,t)$ is an $\cR^2$ valued function and 
$\nabla H=(H_u,H_v)^t$ is the gradient of a $C^3$ homogeneous function 
$H:\cR^2\rightarrow \cR$, the superscript $^t$ represents transpose. The equations are the same as in (\ref{main0}), except that we have simplified the system by letting $c_0=1$ and $A=I_{2\times 2}$ for sake of clarity. The techniques used in this work apply to cases with more general $c_0$ or $A$. In this note we focus on the stability results. We refer to \cite{Hak03, Per00} for global well-posedness of (\ref{main1}), or \cite{Per05, Hak03} for instability results.

The functional space is $L^2(\cR)$, in which the scalar operator $\cL$ defined below has domain $H^2(\cR)$. While in the coupled system, the natural product spaces $L^2(\cR)\times L^2(\cR)$ and $H^2(\cR)\times H^2(\cR)$ are used. The operator $\cL$ sends $H^1(\cR)$ to $H^{-1}(\cR)$; $<\cL u, u>$ denotes the the pairing of $\cL u$ and $u$ on space $H^{-1}(\cR)$ and $H^1(\cR)$. We shall not further dig into the issues on space setting, instead we simply assume all the functions here have enough regularity for all the operations under consideration. However, precise functional settings are crucial in both the well-posedness and the stability proofs, for instance, in case of the space decomposition \cite{BSS87}.

We assume here the homogeneous function $H(\cdot,\cdot)$ is of order $p+2$ with integer $p\geq 1$. 
Suppose $\Phi=(\phi, \mu\phi)^t$ is a {\it proportional solitary wave solution} of (\ref{main1}), where $\mu\in \cR$ is the proportional coefficient. Assume $H_u(1,\mu)>0$ and 
let $$\cM={1\over (p+1)H_u(1,\mu)}\left( \begin{array}{cc} H_{uu}(1,\mu) & H_{uv}(1,\mu) \\
           H_{uv}(1,\mu) & H_{vv}(1,\mu) \end{array}\right).$$

Our main result is 
\begin{theorem}\label{mainthm}
Assume $\det(\cM)<{1\over p+1}$. We have
\begin{enumerate}
\item for $p\leq 4$, $\Phi$ is stable;
\item if $p>4$, then there exists a $\omega_p>1$ such that $\Phi$ is stable for $\omega>\omega_p$ and unstable for $1<\omega<\omega_p$.
\end{enumerate}
\end{theorem}
\begin{remark}Secret-telling moments.
\begin{enumerate} 
\item In this note, we only consider the proportional solitary wave solution. Nevertheless, as the study of nonlinear Schr$\ddot{o}$dinger system suggests \cite{Yang97}, non-proportional solitary wave solutions do exist and are of great importance in real applications. 
\item Given $H$, one can easily compute $\det(\cM)$ and determine the stability of a solitary wave. This idea was first brought up in \cite{BCK15}. The accurate information on eigenvalues is the key to discover the condition under which the stability is achieved.
\end{enumerate}
\end{remark}

\subsection{Strategy}

Since Benjamin \cite{Ben72} set up the rigid mathematical framework, with refinements, generalization and elaboration along the way \cite{Bona75, SS85, Weinstein86, Weinstein87, GSS87, ABH87, BSS87, PW94} etc., stability study in this direction has developed into a sophisticate but ``routine" process. On the one hand, it has close relations with areas such as scattering/inverse scattering theory, spectrum theory, hamiltonian systems, etc., see \cite{Pava09}; on the other hand, the classical method more or less follows a routine process in proving the stability of a solitary wave solution, $\phi$ for instance: 
\begin{itemize}
\item formulate the proper metric $||\cdot||$ to measure the shape deviation;
\item capture a pair of conserved quantities to form a Lyapunov functional $L$, which is invariant too; 
\item for any solution $u$ with perturbation $v=u-\phi$, investigate $L(u)-L(\phi)$ and show there exist positive constants $C_1, C_2$ such that for all time,
\begin{equation}\label{mainineq1}
C_1||v||^2+o(||v||^2)\leq L(u)-L(\phi)\leq C_2||v||^2+o(||v||^2).
\end{equation}
\end{itemize}

Then the second inequality at time $t=0$ implies that $L(u)-L(\phi)$ is bounded by initial deviation $||v(0)||^2=||u(0)-\phi||^2)$, at least when $v$ is small; the first inequality at all later time implies all deviations $||v||^2$ are bounded by $L(u)-L(\phi)$. The fact that $L(u)-L(\phi)$ is invariant gives the stability: the initial deviation (shape difference) bounds the future deviation, at least when the initial deviation is small. 

This clear picture does not show up without subtle issues. The global well-posedness needs to be ready; the key player-a solitary wave solution has to exist; to maintain an effective norm, the infimum of the deviation has to happen within a finite translation\cite{Bona75}; how to transfer a constrained result to a general one,  etc. But the most delicate part is the establishment of inequality (\ref{mainineq1}).

In estimating $L(u)-L(\phi)$, a self-adjoint operator $\cL$ shows up. Roughly, we have $L(u)-L(\phi)=<\cL v, v>+o(||v||^2)$, where $<\cdot,\cdot>$ denotes the adjoint pair on $H^{-1}(\cR)\times H^1(\cR)$. The upper bound, $L(u)-L(\phi)\leq C_2||v||^2+o(||v||^2)$, is easy to show, since $\cL$ is bounded operator from $H^1(\cR)$ to $H^{-1}(\cR)$. The lower bound requires certain amount of ``positiveness'' of $\cL$. The operator is generally not positive definite, but almost so: the residual spectrum is missing since the operator is self-adjoint; it has strictly positive continuous spectrum; it has a simple zero eigenvalue and a unique, simple, negative eigenvalue.  Thus, the space under investigation is able to be decomposed into a direct product of certain subspaces and the major part of the positiveness is kept, so is the positiveness of the lower order term $||v||^2$. 

This classical method calls for a pack of techniques that attach to the analysis of such an operator. To avoid these difficulties, the concentration-compactness method has been developed by Cazenave and Lions from the calculus of variation viewpoint, see \cite{CL82, Lions84}.

We shall embrace the tradition and follow the process mentioned above. But as has been remarked before, we shall omit some ``critical but routine" steps and focus on the those parts which reflect the improvements we provide: the accurate spectrum analysis on the Schr$\ddot{o}$dinger operator; the calculation of a critical quantity $d''(\omega)$. Some examples are given in the end.

\subsection{Preliminaries}

\subsubsection{Point spectrum of Schr$\ddot{o}$dinger operator}
The result in this work relies on the precise information on the point spectrum of a specific Schr$\ddot{o}$dinger operator,
$\cL_0=-\Delta+V(x)$ in one dimensional space, where the Laplacian $\Delta={d^2\over dx^2}$ and $V$ is a potential function. In the special case where the potential function $V(x)=-\alpha\sech^2(x), \alpha>0$, the eigenvalue problem can be solved explicitly; see Landau etal. \cite{LL77}, page 73-74. Here for sake of completeness, we reproduce some details in this aspect.

On space $L^2(\cR)$, this $\phi(x)$ make $\cL_0$ a self-adjoint operator. Let $\sigma$ denote the full spectrum and let $\sigma_p, \sigma_c, \sigma_r$ be the point spectrum, the continuous spectrum and residual spectrum, respectively. For self-adjoint operator the residual spectrum is missing, we have $\sigma(\cL_0)=\sigma_p(\cL_0)\cup\sigma_c(\cL_0)$.
Moreover, 

\begin{lemma}\label{pavalemma0}[Pava \cite{Pava09}] we have
\begin{enumerate}
\item $-{d^2\over dx^2}$ has spectrum $\sigma=\sigma_c=[0,\infty)$.
\item $-{d^2\over dx^2}+\omega$ has spectrum $\sigma=\sigma_c=[\omega,\infty)$.
\item $\cL_0=-{d^2\over dx^2}+\omega-\alpha\sech^2(x)$, being self-adjoint on $L^2(\cR)$, has only real spectrum, which consists of the continuous part $[\omega,\infty)$, together with a finite number of discrete eigenvalues in $(-\infty,\omega)$.
\end{enumerate}
\end{lemma}
To find the eigenvalues of $\cL_0$, we shall solve
\begin{eqnarray}\label{maineigval}
-{d^2\over dx^2}\psi(x)-\alpha~ \sech^2(x)\psi(x)=-\lambda\psi(x).
\end{eqnarray}
Make variable transformation $\xi=\tanh(x)$, hence $\xi\in(-1,1)$ for $x\in(-\infty,\infty)$, we have the the associated Legendre equation for transformed function $ \tilde{\psi}(\xi)\equiv\psi(\tanh^{-1}(\xi))$:
\begin{eqnarray}\label{eigval1}
{d\over d\xi}\Big[(1-\xi^2){{d\over d\xi}}\Big]\tilde{\psi}(\xi)+\Big[\alpha-{\lambda\over 1-\xi^2}\Big]\tilde{\psi}(\xi)=0.
\end{eqnarray}

Note that we work in space $L^2(\cR)$, which require $\psi(x)\rightarrow 0, \mbox{ as } x\rightarrow \pm \infty$, or $$\tilde{\psi}(\xi)\rightarrow 0, \mbox{ as } \xi\rightarrow \pm 1.$$

Letting 
\begin{equation}\label{es}
\epsilon=\sqrt{\lambda}, ~ ~s(s+1)=\alpha, ~ s>0, (\mbox{hence } s={1\over 2}\Big(-1+\sqrt{1+4\alpha}
\Big))
\end{equation} and putting the equation in form of 
\begin{eqnarray}\label{eigval2}
{d\over d\xi}\Big[(1-\xi^2){{d\over d\xi}}\Big]\tilde{\psi}(\xi)+\Big[s(s+1)-{\epsilon^2\over 1-\xi^2}\Big]\tilde{\psi}(\xi)=0,
\end{eqnarray}
we can solve the eigenvalue problem (\ref{eigval2}) completely:
\begin{eqnarray}\label{eigpair1}
\lambda_n&=&-\epsilon_n^2,~   ~\epsilon_n=s-\lceil s\rceil+n, ~ n=1, 2, \cdots, \lceil s\rceil, \nonumber\\
\tilde{\psi}_n(\xi)&=&P_s^{\epsilon_n}(\xi),
\end{eqnarray}
where $\lceil s\rceil$ is the smallest integer greater than or equal to $s$, and $P_s^{\epsilon_n}(\xi),$ is are {\it the associated Legendre functions of the first kind}, meaning the functions are defined within the unit circle in the complex plane. By analytical continuation, Legendre functions of second and third kind can be defined outside of the unit circle. But they do not interest us here.

When $s$ is a positive integer, $\epsilon=1, 2,\cdots, s$ and $P_s^\epsilon$ becomes {\it the Legendre polynomial} which has an explicit form, for example when $s=3$
\begin{eqnarray}
P_3^1(\xi) &= \tfrac{1}{2}(1-\xi^2)^{1/2} (15\xi^2-3), \nonumber\\ 
P_3^2(\xi) &= 15(1-\xi^2)\xi, \nonumber\\ 
P_3^3(\xi) &= 15 (1-\xi^2)^{3/2}.\nonumber 
\end{eqnarray}
The solutions for the corresponding eigenvalue problem (\ref{maineigval}) are
\begin{eqnarray}
\lambda_1&=&-1,\nonumber\\
\psi_1(x)&=&P_3^1(\tanh(x))= 3/2 (1-5 \tanh^2(x)) \sech(x); \nonumber\\\nonumber \\
\lambda_2&=&-4,\nonumber\\
\psi_2(x)&=&P_3^2(\tanh(x))= 15 \tanh(x) \sech^2(x); \nonumber\\ \nonumber\\
\lambda_3&=&-9,\nonumber\\
\psi_3(x)&=&P_3^3(\tanh(x))= -15 \sech^3(x). \nonumber
\end{eqnarray}

%

When $s$ is not an integer, the expressions of $P_s^\epsilon(\xi)$ can be put in form of hypergeometric function (page 74, \cite{LL77}), 
\begin{equation}\label{hyperg1}
P_s^{\epsilon_n}(\xi)=(1-\xi^2)^{{\epsilon_n}\over 2}F[\epsilon_n-s,\epsilon_n+s+1,\epsilon_n+1,{1\over 2}(1-\xi)],
\end{equation}
where the hyper-geometric function $F$ is defined as
\begin{equation}\label{hyperg2}
F[-m,\beta,\gamma,z]={z^{1-\gamma}(1-z)^{\gamma+m-\beta}\over\gamma(\gamma+1)\cdots(\gamma+m-1)}{d^m\over dz^m}[z^{\gamma+m-1}(1-z)^{\beta-\gamma}].
\end{equation}

Again here what concerns us is the information of the concrete eigenvalues, rather than the eigenfunctions.

Another result will be used later.
\begin{lemma}[Theorem B.61,\cite{Pava09}]\label{pavalemma}
Suppose that $\phi\in L^2(\cR)$ satisfies the differential equation $$-\phi''+c\phi-f(\phi)=0,$$
with $c > 0$ and $\phi'$ having exactly one unique zero. Then the differential operator
$$L_0=-{d^2\over dx^2}+[c-f'(\phi)]$$
defined in $L^2(\cR)$ has exactly one simple negative eigenvalue $\lambda_0$; the eigenvalue $0$ is simple with associated eigenfunction $\phi'$; and there exists $\delta>0$ such that every $\lambda\in\sigma(L_0)-\{\lambda_0,0\}$ satisfies $\lambda>\delta$.
\end{lemma}

\subsubsection{Scalar equation}

In this part, by comparing with the scalar equation, we single out the difficulties brought up by the coupling mechanism and focus on these new issues. 

Consider a {\it scalar} BBM equation 
\begin{equation}\label{scalar}
u_t+u_x-u_{xxt}+\hbar u^pu_x=0,
\end{equation}
where $\hbar>0$ is a constant. There exist infinitely many conserved quantities, two of which play an important role in our analysis:
\begin{eqnarray}\label{consv0}
\Omega_0(u)&=&{1\over 2}\int_{\cR} u^2+u_x^2dx, ~\nonumber\\
\Theta_0(u)&=&-{1\over 2}\int_{\cR} u^2+{2\hbar\over (p+2)(p+1)}u^{p+2}dx.
\end{eqnarray}

We assume the solitary wave solution exists in form of $\phi_\omega(x-\omega t)$. Then let $z=x-\omega t, \phi(z)=\phi_\omega(x-\omega t)$, substitute into (\ref{scalar}), we have $$-\omega \phi'+\phi'+\omega\phi'''+({\hbar\over p+1}\phi^{p+1})'=0.$$

Requiring $\phi(z)$ decreases to zero at infinity, we have equation associated with (\ref{scalar})
\begin{equation}\label{scalar1}
-(\omega-1) \phi+\omega\phi''+{\hbar\over p+1}\phi^{p+1}=0.
\end{equation}
When $\omega>1$, the solution to (\ref{scalar1}), which is
$$\phi(z)=\Big[{(p+2)(\omega-1)\over 2\hbar} \sech^2({p\over 2}\sqrt{\omega-1\over\omega}z)\Big]^{1\over p},$$ generates the solitary wave solution $\phi_\omega(x-\omega t)$ to (\ref{scalar}).

Taking derivative of (\ref{scalar1}), we have 
$$-(\omega-1) \phi'+\omega\phi'''+\hbar\phi^p\phi'=0,$$
from which we conclude the operator 
$$\cL_{sclr}=-\omega{d^2\over dx^2}+(\omega-1)-\hbar\phi^p$$
has eigenfunction $\phi'$ corresponding to eigenvalue $0$. 

In view of $\omega>1$ and $\phi$ being in shape of $[\sech^2(\cdot)]^{1\over p}$,  the above is in fact the consequence of Lemma \ref{pavalemma}. 
Furthermore, we conclude 

1. $\cL_{sclr}$ has exactly one simple negative eigenvalue $\lambda_0$; 

2. the eigenvalue $0$ of $\cL_{sclr}$ is simple; 

3. and there exists $\delta>0$ such that every $\lambda\in\sigma(\cL_{sclr})-\{\lambda_0,0\}$ satisfies $\lambda>\delta$.

Let $$d(\omega)=\Theta_0(\phi_\omega)+\omega\Omega_0(\phi_\omega).$$
The classical result is, the main inequality (\ref{mainineq1}) will be established if $$d''(\omega)>0$$  holds, however, under the constraint 
$$\Omega_0(u)=\Omega_0(\phi_\omega)= Const.$$
Then an extra argument (omitted in this note) shall bridge this constrained case to the general case and Theorem \ref{mainthm} would then be completed.

\begin{remark}
For a general solitary wave solution, there is a requirement on the existence of an interval of $\omega$, on which $d(\omega)$ has certain smoothness. It is not a problem for the special solitary solution $\phi$ here.

\end{remark}

Four the coupled system (\ref{main1}), we shall walk through the same procedure, however, with conserved quantities and operator in different forms. And the main difficulty comes from the spectrum analysis of the new operator.

\subsubsection{Coupled system}

Consider (\ref{main1}), $$U_t+U_x-U_{xxt}+(\nabla H(U))_x=0,$$ 
We have conserved quantities
$$\Omega(U)={1\over 2}\int_{\cR}U^tU+U_x^tU_xdx,$$
$$\Theta(U)=-{1\over 2}\int_{\cR}U^tU+2H(U)dx.$$

We consider the proportional solitary wave solution $U(x,t)=\Phi_\omega(x-\omega t)=(\phi_\omega(x-\omega t), \psi_\omega(x-\omega t))^t$ of system (\ref{main1}), where $\psi_\omega=\mu\phi_\omega$ with $\mu$ a real constant.
Later on, we often write $\phi, \psi$ for $\phi_\omega(x-\omega t), \psi_\omega(x-\omega t)$ when it is in a clear context.

Substituting into the system (\ref{main1}) and demanding a solitary solution, we have 
$$H_v(1,\mu)=\mu H_u(1,\mu),$$ and
\begin{equation}\label{soli1}
-\omega\phi''+(\omega-1)\phi-H_u(1,\mu)\phi^{p+1}=0.
\end{equation}

By (\ref{scalar1}), the solution of (\ref{soli1}) is found to be
\begin{equation}\label{phi}
\phi=\Big[{(p+2)(\omega-1)\over 2H_u(1,\mu)} \sech^2({p\over 2}\sqrt{\omega-1\over\omega}x)\Big]^{1\over p}.
\end{equation}

Linearizing (\ref{main1}) around $\Phi=(\phi, \mu\phi)^t$, we have 
$$\cL=\Big(-\omega{d^2\over dx^2}+(\omega-1)\Big)I_{2\times 2}-\nabla^2H_\omega,$$
where $I_{2\times 2}=\left( \begin{array}{cc} 1 & 0 \\
           0& 1 \end{array}\right) \mbox{ and } \nabla^2H_\omega=\left( \begin{array}{cc} H_{uu}(\phi,\psi) & H_{uv}(\phi,\psi) \\
           H_{uv}(\phi,\psi) & H_{vv}(\phi,\psi) \end{array}\right).$\\ \\
\textbf{ Now our mission in this work is to }

\begin{enumerate}
\item analyze $\cL$ and prove 
\begin{itemize}\label{item1}
\item it has exactly one simple negative eigenvalue $\lambda_-$; 
\item the eigenvalue $0$ of $\cL$ is simple; 
\item the continuous spectrum of $\cL$ is bounded below by a positive number.
\end{itemize}

\item show $d''(\omega)>0$ for $d(\omega)=\Theta(\Phi_\omega)+\omega\Omega(\Phi_\omega)$.
\end{enumerate}

\section{Analysis on $\cL$}
\subsection{Eigenvalue and unitary transformation}
We have
$$\nabla^2H_\omega=\left( \begin{array}{cc} H_{uu}(\phi,\psi) & H_{uv}(\phi,\psi) \\
           H_{uv}(\phi,\psi) & H_{vv}(\phi,\psi) \end{array}\right)$$
$$=\phi^p\left( \begin{array}{cc} H_{uu}(1,\mu) & H_{uv}(1,\mu) \\
           H_{uv}(1,\mu) & H_{vv}(1,\mu) \end{array}\right)$$
$$=\phi_0\cM,$$ 
with $$\phi_0={{1\over 2}(p+1)(p+2)(\omega-1)} \sech^2({p\over 2}\sqrt{\omega-1\over\omega}x),$$ and
\begin{equation}\label{cM}
\cM={1\over (p+1)H_u(1,\mu)}\left( \begin{array}{cc} H_{uu}(1,\mu) & H_{uv}(1,\mu) \\
           H_{uv}(1,\mu) & H_{vv}(1,\mu) \end{array}\right).
\end{equation}

We expect the existence of a unitary transformation, by which we can simplify $\cM$ so that it is easier to analyze the spectrum, noting the unitary transformation does not change the spectrum of an operator. Since $\cM$ is real symmetric, it suffices to find the eigenvalues of $\cM$; if we can find two eigenvalues of $\cM$, we find the unitary transformation matrix. 
           
\begin{lemma}
The matrix $\cM$ has two eigenvalues, which are $\lambda_1=1, \lambda_2=\det \cM$.
\end{lemma}           
\begin{proof} From (\ref{cM}), we get
$$\det(\cM-\lambda I)={1\over (p+1)^2H^2_u(1,\mu)}\det(\tilde{\cM}),$$
$$\det(\tilde{\cM})=\det\left( \begin{array}{cc} H_{uu}-\lambda (p+1)H_u & H_{uv} \\
           H_{uv} & H_{vv}-\lambda(p+1) H_u \end{array}\right).$$
We use the following result which shall be proved presently in Lemma \ref{hessian1}
$$H_{uu}+\mu H_{uv}=(p+1) H_u,$$
$$H_{uv}+\mu H_{vv}=(p+1)\mu H_u.$$
Multiplying the second column of $\tilde{\cM}$ and adding to the first one, we have
$$\det(\cM-\lambda I)={(1-\lambda)\over (p+1)H_u(1,\mu)}\det\left( \begin{array}{cc}  1& H_{uv} \\
           \mu & H_{vv}-\lambda(p+1)H_u \end{array}\right).$$
Solving $\det(\cM-\lambda I)=0$, we get $\lambda_1=1, \lambda_2=\det(\cM)/\lambda_1$, as the lemma states.
\end{proof}       
\begin{lemma}\label{hessian1}
 For $H$ defined above, we have
$$H_{uu}(1,\mu)+\mu H_{uv}(1,\mu)=(p+1) H_u(1,\mu),$$ and
$$H_{uv}(1,\mu)+\mu H_{vv}(1,\mu)=(p+1)\mu H_u(1,\mu).$$
\end{lemma}
\begin{proof}
Assume $H(u,v)=\sum\limits_{i=0}^{p+2}C_iu^iv^{p+2-i}$. The result can be easily verified by a few differentiations and substitutions.
\end{proof}           

Now we are able to form an orthogonal matrix $O$, by the normalized eigenvectors, such that
$$O^t\cM O=\left( \begin{array}{cc} 1 & 0 \\
0 & \det(\cM) \end{array}\right),$$
           hence           
           $$O^t\cL O=\Big(-\omega{d^2\over dx^2}+(\omega-1)\Big)I_{2\times 2}-\phi_0\left( \begin{array}{cc} 1 & 0 \\
0 & \det(\cM) \end{array}\right),$$
with $\phi_0(z)={(p+1)(p+2)(\omega-1)\over 2} \sech^2({p\over 2}\sqrt{\omega-1\over\omega}z).$

\subsection{Spectrum analysis}
We want to find out under what condition on $\cM$, operator $\cL$ satisfies the spectrum conditions stated at the end of the previous section.

Since $O$ is an orthogonal matrix, it suffices to study $\tilde{\cL}\equiv O^t\cL O$.
We want to prove
\begin{theorem}\label{mainthm2}
$\tilde{\cL}$ is a self-adjoint operator on $L^2(\cR)\times L^2(\cR)$. 
It has a unique simple negative eigenvalue when $\det(\cM)<{1\over p+1}$;
$0$ is the only other eigenvalue of $\tilde{\cL}$, which is simple; the continuous spectrum of $\tilde{\cL}$ is bounded 
from below by a positive number.
\end{theorem}

We shall complete the proof of Theorem \ref{mainthm2} by showing a series of lemmas in the following.

The self-adjointness of $\tilde{\cL}$ is straightforward, since it is a compact perturbation of a selfadjoint operator \cite{Pava09}. Denote $\tilde{\cL}=\left( \begin{array}{c} \cL_1 \\ \cL_2 \end{array} \right)$ with 
$$\cL_1=-\omega{d^2\over dx^2}+(\omega-1)-\phi_0,$$
$$\cL_2=-\omega{d^2\over dx^2}+(\omega-1)-\det(\cM)\phi_0.$$           

First, observe that
$$\cL_1\psi_1=\lambda_1\psi_1\Rightarrow \tilde{\cL}\left( \begin{array}{c} \psi_1 \\ 0 \end{array} \right)=\lambda_1\left( \begin{array}{c} \psi_1 \\ 0 \end{array} \right),$$
$$\cL_2\psi_2=\lambda_2\psi_2\Rightarrow \tilde{\cL}\left( \begin{array}{c} 0\\ \psi_2  \end{array} \right)=\lambda_2\left( \begin{array}{c} 0\\ \psi_2  \end{array} \right).$$
Namely, any possible negative (or $0$) eigenvalue of $\cL_1$ or $\cL_2$ would make a (negative or $0$) eigenvalue of $\cL$. 

We will show presently that $\cL_1$ has a unique simple negative eigenvalue and a simple $0$ eigenvalue, the only possibility for $\cL$ to share the same property is that $\cL_2$ has strictly positive spectrum.

\subsubsection{Analysis on $\cL_1$} 
          
Recall $\phi_0={(p+1)(p+2)(\omega-1)\over 2} \sech^2({p\over 2}\sqrt{\omega-1\over\omega}x)$, hence 
$$\cL_1=-\omega{d^2\over dx^2}+(\omega-1)-{(p+1)(p+2)(\omega-1)\over 2} \sech^2({p\over 2}\sqrt{\omega-1\over\omega}x)$$ is a compact perturbation of a self-adjoint operator and it is self-adjoint too.
Also note this is exactly the linear operator we have by differentiating (\ref{soli1}), where we got a byproduct, namely the eigenpair of $\cL_1$:

$
 \left\{ \begin{array}{l}
         \mbox{eigenvalue}: 0,\\
         \mbox{eigenfunction}: \phi'(x), \mbox{ where } \phi \mbox{ is from  (\ref{phi})}. 
                \end{array}\right.
$ 

Note that $\phi$ has shape of $[\sech^2(x)]^{1\over p}$. Lemma \ref{pavalemma} applies to ${1\over \omega}\cL_1$ (hence to $\cL_1$ since $\omega>1$). Therefore, we have shown
\begin{lemma}
The self-adjoint operator $\cL_1$ has a unique negative simple eigenvalue and another simple eigenvalue $0$. 
 \end{lemma}
 Moreover, we can compute the exact negative eigenvalue of $\cL_1$.

Let $y={p\over 2}\sqrt{\omega-1\over\omega}x$, then  
$$\cL_1=-\omega{d^2\over dx^2}+(\omega-1)-\phi_0$$           
$$=-\omega{p^2(\omega-1)\over 4\omega}{d^2\over dy^2}+(\omega-1)-\phi_0$$                      
$$={p^2(\omega-1)\over 4}\Big(-{d^2\over dy^2}+{4\over p^2}-{2(p+1)(p+2)\over p^2}\sech^2(y)\Big).$$  

It is easy to see that studying spectrum of $\cL_1$ is equivalent to studying that of 
$$\cL_{1y}=-{d^2\over dy^2}+{4\over p^2}-{2(p+1)(p+2)\over p^2}\sech^2(y).$$
And it suffices to investigate 
$$\cL_{1y0}=-{d^2\over dy^2}-{2(p+1)(p+2)\over p^2}\sech^2(y),$$ in which case the previous  knowledge of Schr$\ddot{o}$dinger operator comes into play.

We know that $\cL_{1y}$ has a unique negative eigenvalue. $\cL_{1y0}$ might have more than one negative eigenvalue, but only the least one of $\cL_{1y0}$ corresponds to the unique negative one of $\cL_{1y}$, with difference ${4\over p^2}$. So we focus on calculating the least negative eigenvalue of $\cL_{1y0}$.

Comparing $\cL_{1y0}$ with (\ref{maineigval}) and (\ref{es}), we have 
$$\alpha={2(p+1)(p+2)\over p^2},$$
$$\alpha=s(s+1)\Rightarrow s={-1+\sqrt{1+{8(p+1)(p+2)\over p^2}} \over 2}={p+2\over p}.$$

By  (\ref{eigpair1}), the least negative eigenvalue corresponds to the largest $\epsilon_n$, which is $s$ when $n=\lceil s\rceil$. Namely we have the negative eigenpair:
$$\lambda_{1y0}=-\epsilon_n^2=-s^2=-(1+{2\over p})^2,$$
$$\psi_{1y0}=P_s^s(\xi=\tanh(y))=\sech^sy.$$           
The eigenfunction comes from (\ref{hyperg1}) and (\ref{hyperg2}) with $\epsilon_n=s, \xi=\tanh(y)$.

Now tracing back to $\cL_{1y}$ and  $\cL_1$, we get their corresponding eigenvalues 
$$\lambda_{1y}=\lambda_{1y0}+{4\over p^2}=-(1+{4\over p}),$$
$$\lambda_1={p^2(\omega-1)\over 4}\lambda_{1y}={-p(p+4)(\omega-1)\over 4}.$$ 

\subsubsection{Analysis on $\cL_2$}

Our goal here is to find conditions under which $\cL_2$ has strictly positive spectrum.
Recall
$$\cL_2=-\omega{d^2\over dx^2}+(\omega-1)-\det(\cM)\phi_0$$           
$$=-\omega{p^2(\omega-1)\over 4\omega}{d^2\over dy^2}+(\omega-1)-\det(\cM)\phi_0$$                      
$$={p^2(\omega-1)\over 4}\Big(-{d^2\over dy^2}+{4\over p^2}-{2(p+1)(p+2)\over p^2}\det(\cM)\sech^2(y)\Big)$$  

\begin{lemma}
The operator $\cL_2$ has strictly positive spectrum if $$\det(\cM)<{1\over p+1}.$$
\end{lemma}

\begin{proof}

It suffices to prove the same result for $$\cL_{2y}=-{d^2\over dy^2}+{4\over p^2}-{2(p+1)(p+2)\over p^2}\det(\cM)\sech^2(y).$$ 

From lemma \ref{pavalemma0}, $\sigma(\cL_{2y})=\sigma_c(\cL_{2y})\cup \sigma_p(\cL_{2y})$ and $\sigma_c(\cL_{2y})=[{4\over p^2},\infty)$. So what we need to do here is to find what makes all the eigenvalues strictly positive. 

Similar to the previous subsection, for 
$$\cL_{2y0}=\cL_{2y}-{4\over p^2}=-{d^2\over dy^2}-{2(p+1)(p+2)\over p^2}\det(\cM)\sech^2(y),$$
we have the least negative eigenvalues $$\lambda_{2y0}=-s^2,$$ where
$$s={-1+\sqrt{1+{8(p+1)(p+2)\over p^2}\det(\cM)} \over 2}.$$

And what we want is $\lambda_{2y}=-s^2+{4\over p^2}>0$, which forces the following
$$s^2<{4\over p^2}\Rightarrow s<{2\over p}$$
$$\Rightarrow -1+\sqrt{1+{8(p+1)(p+2)\over p^2}\det(\cM)}<{4\over p}$$
$$\Rightarrow 1+{{8(p+1)(p+2)\over p^2}\det(\cM)}<({4\over p}+1)^2$$
$$\Rightarrow {{8(p+1)(p+2)\over p^2}\det(\cM)}<{8(p+2)\over p^2}$$
$$\Rightarrow \det(\cM)<{1 \over (p+1)}.$$
We complete Theorem \ref{mainthm2}.
\end{proof}
Next, we will show Theorem \ref{mainthm} by finding out when $d''(\omega)>0$ holds.

\section{Calculations for $d''(\omega)$}

\begin{theorem}\label{mainthm3}
We have the following for $d''(\omega)$:
\begin{enumerate}
\item If $p\leq 4$, then $d''(\omega)>0.$

\item If $p>4$, then there exists a $\omega_p>1$ such that $d''(\omega)>0$ for $\omega>\omega_p$ and $d''(\omega)<0$ for $1<\omega<\omega_p$.
\end{enumerate}
\end{theorem}
\begin{proof}
We know the solitary wave solution $\Phi_{\omega}=(\phi_\omega,\psi_\omega)^t$ is a minimum point for the Lyapunov functional 
$$L(U)=\Theta(U)+\omega\Omega(U).$$
Hence $$L'(\Phi_\omega)=\Theta'(\Phi_\omega)+\omega\Omega'(\Phi_\omega)=0.$$
Recall $$d(\omega)=\Theta(\Phi_\omega)+\omega\Omega(\Phi_\omega).$$
So $$d'(\omega)=<L'(\Phi_\omega), {d\over d\omega}\Phi_\omega>+\Omega(\Phi_\omega)=\Omega(\Phi_\omega).$$
Thus, $$d''(\omega)={d\over d\omega}\Omega(\phi_\omega,\psi_\omega).$$ 
Using $\phi, \psi$ for $\phi_\omega,\psi_\omega$, we shall investigate 
$$\Omega(\phi,\psi)={1\over 2}\int_{\cR} \phi^2+\phi_x^2dx.$$
Recall $\phi=\Big[{(p+2)(\omega-1)\over 2H_u(1,\mu)} \sech^2({p\over 2}\sqrt{\omega-1\over\omega}x)\Big]^{1\over p}$. Let $A={(p+2)(\omega-1)\over 2H_u(1,\mu)}, B={p\over 2}\sqrt{\omega-1\over\omega}$, and we rewrite
$$\phi=A^{1\over p}\Big[\sech^2(Bx)\Big]^{1\over p},$$
$$\phi_x=-{2\over p}A^{1\over p}B~\mbox{sinh(Bx)}\sech^{{2\over p}+1}(Bx).$$

Sine $\psi=\mu\phi$, we have
$$\Omega(\phi,\psi)={1+\mu^2\over 2}\int_{\cR}\phi^2+\phi_x^2dx$$
$$={1+\mu^2\over 2}\Big[A^{2\over p}B^{-1}\int_{\cR} \sech^{4\over p}(y)dy+{1\over p^2}A^{2\over p}B\int_{\cR}\mbox{sinh}^2(y)\sech^{{4\over p}+2}(y)dy\Big]$$

$$=\vartheta_1(\omega-1)^{2\over p}\sqrt{\omega\over\omega-1}+\vartheta_2(\omega-1)^{2\over p}\sqrt{\omega-1\over\omega},$$
where $\vartheta_1>0, \vartheta_2>0$ are constants.

So $$d''(\omega)={d\over d\omega}\Omega(\phi,\psi)$$
$$=(\omega-1)^{{2\over p}-{3\over 2}}\omega^{-{3\over 2}}q(\omega),$$ where $q(\omega)$ is a quadratic polynomial
$$q(\omega)={2\over p}(\vartheta_1+\vartheta_2)\omega^2-({\vartheta_1\over 2}+{2\vartheta_2\over p}-{\vartheta_2\over 2})\omega-{\vartheta_2\over 2}.$$

We have trivially  $q(1)=2\vartheta_1({1\over p}-{1\over 4})$ and $q'(\omega)={4\over p}(\vartheta_1+\vartheta_2)\omega-({\vartheta_1\over 2}+{2\vartheta_2\over p}-{\vartheta_2\over 2})$.

So when $p\leq 4$, we have $q(1)\geq 0, q'(\omega)>0$, hence $q(\omega)>0$ and $d''(\omega)>0$ for $\omega>1$.

If $p>4$, then $q(1)<0$. For quadratic polynomial $q(x)$ with positive leading coefficient ${2\over p}(\vartheta_1+\vartheta_2)$, there is zero $c_p>1$ exists such that $q(\omega)>0$ for $\omega>\omega_p$ and $q(\omega)<0$ for $1<\omega<\omega_p$.
This concludes the lemma.
\end{proof}

Theorem \ref{mainthm} is concluded in view of Theorem \ref{mainthm2} and Theorem 
\ref{mainthm3}.

\section{Examples}

\begin{example}\label{example1}
$$u_t+u_x-u_{xxt}+b_1u^pu_x+b_2(u^pv)_x+b_2v^pv_x=0,$$
$$v_t+v_x-v_{xxt}+b_1v^pv_x+b_2u^pu_x+b_2(uv^p)_x=0.$$
This is an example constructed in \cite{Per05}, with $c_0=1, a_1=1, a_2=0$. In this case, $$H(U)=b_1{u^{p+2}\over (p+1)(p+2)}+b_1{v^{p+2}\over(p+1)(p+2)}+b_2{u^{p+1}\over p+1}v+b_2u{v^{p+1}\over p+1}.$$ 
In \cite{Per05}, the author consider the proportional solitary wave with $\mu=1$, finding the conditions for stability as
\begin{equation}\label{hak1}
-b_2(p+2)<b_1\leq -b_2(p-2),
\end{equation}
\begin{equation}\label{hak2}
-b_2(p+2)<b_1\leq -b_2{(p+2)(p^2-p-4)\over p^2+3p}.
\end{equation}
The method in this note gives  a stronger result; the upper bound is relaxed.
For the corresponding case where $\mu=1$, using our method, one finds
 $$\cM={1\over (p+1)H_u(1,\mu)}\left( \begin{array}{cc} H_{uu}(1,\mu) & H_{uv}(1,\mu) \\
           H_{uv}(1,\mu) & H_{vv}(1,\mu) \end{array}\right)$$
 $$={1\over b_1+(p+2)b_2}\left( \begin{array}{cc} b_1+pb_2 & 2b_2 \\
           2b_2 & b_1+pb_2 \end{array}\right).$$

$$\det(\cM)<{1\over p+1}\Rightarrow (p+1)[(b_1+pb_2)^2-4b_2^2]<(b_1+(p+2)b_2)^2$$
$$\Rightarrow (p+1)[(b_1+(p-2)b_2][(b_1+(p+2)b_2]<(b_1+(p+2)b_2)^2$$
$$\Rightarrow (p+1)[(b_1+(p-2)b_2]<(b_1+(p+2)b_2)$$
$$\Rightarrow pb_1\leq-(p+1)(p-3)b_2$$
$$\Rightarrow b_1\leq-{(p+1)(p-3)\over p}b_2.$$
To get the third inequality from bottom, the assumption $H_u=b_1+(p+2)b_2>0$ is used, which gives 
a lower bound for $b_1$, same as (\ref{hak1}) and (\ref{hak2}).
To see that our result is stronger, first note that $b_2>0$ holds true by comparing the first and third 
terms in inequality (\ref{hak1}) (or (\ref{hak2})).
Then, it is easy to verify that the upper bounds given here are larger than those given 
in (\ref{hak1}) and (\ref{hak2}).
\end{example}

\begin{example}\label{example2} Let $q\geq 1$ be an integer. Consider
$$u_t+u_x-u_{xxt}+(u^qv^{q+1})_x=0,$$
$$v_t+v_x-v_{xxt}+(u^{q+1}v^q)_x=0.$$

This is an example considered in \cite{Hak03}, where the result (Lemma 3.1) requires $\omega-1+H_{uv}-H_{u^2}>0$, which is obviously a much stronger assumption than ours. 
This is exactly the point: by taking advantage of the accurate point spectrum information, we weaken conditions like this.

By our result $$H(u,v)={u^{q+1}v^{q+1}\over q+1}.$$ 
Existence of solitary wave solution forces $\mu=1$. This case corresponds to $p=2q-1$. We have here
 $$\cM={1\over (p+1)H_u(1,\mu)}\left( \begin{array}{cc} H_{uu}(1,\mu) & H_{uv}(1,\mu) \\
           H_{uv}(1,\mu) & H_{vv}(1,\mu) \end{array}\right)$$
 $$={1\over 2q^2}\left( \begin{array}{cc} q(q-1) & q(q+1) \\
           q(q+1)&q(q-1) \end{array}\right).$$
           Hence
$$\det(\cM)=-{1\over q}<0<{1\over p+1}={1\over 2q}.$$ The condition is satisfied, so Theorem \ref{mainthm} applies.
\end{example}

\begin{example}\label{example3} Let $q\geq 1$ be an integer. Let $A, B\neq 0$, consider the decoupled system
$$u_t+u_x-u_{xxt}+(Au^{p+1})_x=0,$$
$$v_t+v_x-v_{xxt}+(Bv^{p+1})_x=0.$$
In this case $$H(u,v)={Au^{p+2}+Bv^{p+2}\over q+2}.$$ 
$$H_v(1,\mu)=\mu H_u(1,\mu)\Rightarrow B\mu^{p+1}=A\mu$$
$$\Rightarrow \mu=0, \mbox{ or } B\mu^p=A.$$

 $$\cM={1\over (p+1)H_u(1,\mu)}\left( \begin{array}{cc} H_{uu}(1,\mu) & H_{uv}(1,\mu) \\
           H_{uv}(1,\mu) & H_{vv}(1,\mu) \end{array}\right)$$
 $$={1\over (p+1)A}\left( \begin{array}{cc} A(p+1) & 0 \\
           0&B(p+1)\mu^p \end{array}\right).$$
           Hence
$$\det(\cM)={B\mu^p\over A}.$$
If $\mu=0$, then $\det(\cM)=0<{1\over p+1}$. Theorem \ref{mainthm} applies.\\
If $\mu\neq 0$, then $\det(\cM)=1>{1\over p+1}$.
\end{example}

This is a slight generalization of example in \cite{BCK15}. It tells that, the proportional solitary wave solution of a {\it decoupled} BBM system is generally not stable, unless one component is $0$. Indeed, if the solitary wave $(\phi_\omega,\psi_\omega)$ is perturbed in such a way that the initial data are exactly a new pair of solitary waves $(\phi_{\omega_1}, \psi_{\omega_2})$ that are close to $(\phi_\omega,\psi_\omega)$, but $\omega_1\neq \omega_2$, then the solution (new pair of solitary waves) will slip away form $(\phi_\omega,\psi_\omega)$ as time goes on, hence not stable according to the definition we are using here.  This seems due to the proportional requirements. This conclusion more or less goes against intuition. The fact that each component is stable while the combination is unstable actually brings up interesting questions: does there exist {\it coupled} non-proportional solitary wave solutions and if so, are they stable? Is there a different definition of stability that can accommodate the missing stability that solely caused by the phase difference?

In the study of vector solitons in birefringent nonlinear optical fibers, the nonlinear Schr$\ddot{o}$dinger equations have been proved numerically to have non-proportional solitary wave solution \cite{Yang97}. Next, the authors hope to investigate the similar subject in other dispersive equation system such as coupled BBM model.

The final example is an analogy to the one in \cite{Yang97} for a coupled nonlinear Sch$\ddot{o}$rdinger equations.
\begin{example}\label{jianke1}
$$u_t+u_x-u_{xxt}+((u^2+\beta v^2)u)_x=0,$$
$$v_t+v_x-v_{xxt}+((v^2+\beta u^2)v)_x=0.$$
 
In this case $p=2$ and $$H(u,v)={u^4\over 4}+{\beta u^2v^2\over 2}+{v^4\over 4}.$$ 
$$H_v(1,\mu)=\mu H_u(1,\mu)\Rightarrow (\beta-1)\mu(1-\mu^2)=0.$$\\
If $\beta=1$, then $\det(\cM)={1\over 3}={1\over p+1}$.\\
If $\beta\neq 1, \mu=0$, then $\det(\cM)={\beta\over 3}<{1\over p+1}\Rightarrow\beta<1$.\\
If $\beta\neq 1, \mu\neq 0$, then $\mu^2=1, \det(\cM)<{1\over 3}\Rightarrow \beta^2>1$.
\end{example}

\section*{Acknowledgments} The authors would like to thank Professor Irena Lasiecka for initializing and stimulating their joint work. The second author is grateful for Professor Irena's kind hospitality and warm encouragement for all time.


\medskip
\medskip


\begin{thebibliography}{99}

\bibitem{AAM99}
     \newblock E. Alarcon, J. Angulo and J.F. Montenegro, 
     \newblock Stability and instability of solitary waves for a nonlinear dispersive system,
     \newblock \emph{Nonlinear Anal. Ser. A: Theory Methods}, \textbf{36} (1999), no.8, 1015--1035.


\bibitem{ABH87} \newblock J. P. Albert, J. L. Bona and D. B. Henry, 
\newblock Sufficient conditions for stability of solitary-wave solutions of model equations for long waves,
\newblock \emph{Phys. D.} \textbf{24} (1987), 343--366.

\bibitem{Ben72} \newblock T.B. Benjamin, 
\newblock The stability of solitary waves,
\newblock \emph{Proc. Royal Soc. London, Series A}, \textbf {328}, (1972), 153--183.

\bibitem{Bona75} 
\newblock J.L. Bona, On the stability theory of solitary waves, 
\newblock \emph{Proc. Royal Soc. London, Series A}, \textbf{ 344}, (1975) 363--375.

\bibitem{BCK15} 
\newblock J. L. Bona, H. Chen and O. A. Karakashian, 
\newblock Stability of solitary-wave solutions of systems of dispersive equations,
\newblock submitted.

\bibitem{BSS87} 
\newblock J.L. Bona, P.E. Souganidis and W.A. Strauss, 
\newblock Stability and instability of solitary waves of Korteweg-de Vries type, 
\newblock \emph{Proc. Royal Soc. London Ser. A} \textbf{ 411} (1987), 395--412.

\bibitem{CL82} 
\newblock T. Cazenave and P.-L. Lions, 
\newblock Orbital stability of standing waves for some nonlinear Schr$\ddot{o}$dinger equations, 
\newblock \emph{Comm. Math. Phys.} \textbf{ 85} (1982), 549-561.

\bibitem{GSS87} 
\newblock M. Grillakis, J. Shatah and W. Strauss, 
\newblock Water stability theory of solitary waves in the presence of symmetry. I., 
\newblock \emph{J. Funct. Anal.} \textbf{ 74} (1987), no. 1, 160--197. 


\bibitem{Hak03} 
\newblock S. Hakkaev, 
\newblock Stability and instability of solitary wave solutions of a nonlinear dispersive system of Benjamin-Bona-Mahony type, 
\newblock \emph{ Serdica Math. J.} 29 (2003), no. 4, 337--354. 

\bibitem{LL77}
\newblock  L.D. Landau and E.M. Lifshitz, 
\newblock  \emph{ Quantum Mechanics, Non-Relativistic Theory, Third Edition: Volume 3 (Course of Theoretical Physics)},
\newblock  Pergamon Press, 3$^{rd}$ revised,  (1977), p.677.


\bibitem{Lions84} 
\newblock P.-L. Lions, 
\newblock The concentration-compactness principle in the calculus of variations. The locally compact case, part 1, 
\newblock \emph{ Ann. Inst. H. Poincar\'{e}, Anal. Non Lin\'{e}are.} \textbf{ 1} (1984), 109--145.


\bibitem{Pava09} 
\newblock J. Angulo Pava, 
\newblock \emph{ Nonlinear Dispersive Equations. Existence and Stability of Solitary and Periodic Travelling Wave Solutions, Mathematical Surveys and Monographs}, \textbf{ 156}, 
\newblock American Mathematical Society, Providence, RI, (2009),  xii+256 pp.

\bibitem{PW94} 
\newblock R.L. Pego and M.I. Weinstein, 
\newblock Asymptotic stability of solitary waves,
\newblock \emph{ Comm. Math. Phys.}\textbf{ 164}, (1994), 305--349.

\bibitem{Per00} 
\newblock J. M. Pereira,
\newblock Global existence and decay of solutions of a coupled system of BBM-Burgers equations. 
\newblock \emph{ Rev. Mat. Complut.} \textbf{ 13} (2000), no. 2, 423--443.

\bibitem{Per05}
\newblock J.M. Pereira, 
\newblock Stability and instability of solitary waves for a system of coupled BBM equations, 
\newblock \emph{ Appl. Anal.} 84 (2005), no. 8, 807--819.


\bibitem{SS85} 
\newblock J. Shatah and W. Strauss,
\newblock Instability of nonlinear bound states, 
\newblock \emph{ Comm. Math. Phys.} \textbf{ 100} (1985), 173--190.

\bibitem{Weinstein86} 
\newblock M.I. Weinstein,
\newblock Liapunov stability of ground states of nonlinear dispersive evolution equations, 
\newblock \emph{ Comm. Pure Appl. Math.} \textbf{ 39} (1986), 51--68.

\bibitem{Weinstein87}
\newblock M.I. Weinstein, 
\newblock Existence and dynamic stability of solitary wave solutions of equations arising in long wave propagation, 
\newblock \emph{ Comm. PDE} \textbf{ 12} (1987), 1133--1173.

\bibitem{Yang97} 
\newblock J. Yang,
\newblock Vector solitons and their internal oscillations in birefringent nonlinear optical fibers, 
\newblock \emph{ Stud. Appl. Math.} \textbf{ 98} pp.61--97, (1997).




%
%
%
%
%
%
%
%
%
%
%
\end{thebibliography}
\end{document}